\documentclass{article}
\usepackage{amsmath,amscd,amsfonts,amssymb,amsthm}
\usepackage{graphicx}
\usepackage{epstopdf}
\usepackage{a4wide}
\usepackage{graphicx}
\usepackage{subfigure}
\usepackage[latin1]{inputenc}
\usepackage{color}
\usepackage{caption}

\newcommand{\subj}[1]{\par\noindent{\bf Mathematics Subject Classification 2010: }#1.}
\newcommand{\keyw}[1]{\par\noindent{\bf Keywords: }#1.}

\theoremstyle{definition}
\newtheorem{definition}{Definition}
\newtheorem{theorem}{Theorem}
\newtheorem{example}{Example}

\theoremstyle{remark}

\def\a{\alpha}
\def\t{\tau}
\def\p{\psi}

\def\LD{{^CD_{a+}^{\a,\p}}}
\def\RD{{^CD_{b-}^{\a,\p}}}
\def\RLD{{D_{a+}^{\a,\p}}}
\def\RRD{{D_{b-}^{\a,\p}}}
\def\LI{{I_{a+}^{\a,\p}}}
\def\RI{{I_{b-}^{\a,\p}}}

\begin{document}

\title{Optimality conditions for fractional variational problems with free terminal time}

\author{Ricardo Almeida\\
{\tt ricardo.almeida@ua.pt}}

\date{Center for Research and Development in Mathematics and Applications (CIDMA)\\
Department of Mathematics, University of Aveiro, 3810--193 Aveiro, Portugal}

\maketitle


\begin{abstract}
This paper provides necessary and sufficient conditions of optimality for a variational problem involving a fractional derivative with respect to another function. Fractional Euler--Lagrange equations are proven for the fundamental problem and when in presence of an integral constraint. A Legendre condition, which is a second-order necessary condition, is also obtained. Other cases, such as the infinite horizon problem, with delays in the Lagrangian, and with high-order derivatives, are considered. A necessary condition that the optimal fractional order must satisfy is proved.
\end{abstract}

\subj{26A33, 49K05, 	34A08}

\keyw{Fractional calculus, Euler--Lagrange equation, Legendre condition, isoperimetric problem}


\section{Introduction}

In this work, we consider fractional integrals and fractional derivatives with respect to another function $\p$ (see \cite{Samko}).
To fix notation, in the following $\a>0$ is a real, $\p\in C^1[a,b]$ is an increasing function, such that $\p'(t)\not=0$, for all $t\in [a,b]$, and
$x:[a,b]\to\mathbb R$ is a function with some assumptions, so that the fractional operators that we deal with are well defined.
The left fractional integral of $x$, with respect to $\p$ of order $\a$, is defined as
$$\LI x(t):=\frac{1}{\Gamma(\a)}\int_a^t \p'(\t)(\p(t)-\p(\t))^{\a-1}x(\t)\,d\t,$$
and the right fractional integral of $x$ is
$$\RI x(t):=\frac{1}{\Gamma(\a)}\int_t^b \p'(\t)(\p(\t)-\p(t))^{\a-1}x(\t)\,d\t.$$
Considering special cases for the kernel, that is, for the function $\p$, we recover e.g. the Riemann--Liouville, the Hadamard and the
Erd\'{e}lyi--Kober integrals.
For fractional derivatives of Riemann--Liouville type, the left and right fractional derivatives of $x$ are defined as
$$\RLD x(t):= \left(\frac{1}{\p'(t)}\frac{d}{dt}\right)^n {I_{a+}^{n-\a,\p}}x(t)$$
and
$$\RRD x(t):=  \left(-\frac{1}{\p'(t)}\frac{d}{dt}\right)^n {I_{b-}^{n-\a,\p}}x(t)$$
respectively, where $n=[\a]+1$.
In this paper, we deal mainly with a Caputo-type fractional derivative. The concept is similar to the Riemann--Liouville derivative, but the
order of the dual integral/derivative is switched (see \cite{AlmeidaN}).

\begin{definition} Let $x\in C^n[a,b]$ be a function. The left Caputo fractional derivative of $x$ of order $\a$ with respect to $\p$ is given by
$$\LD x(t):=   {I_{a+}^{n-\a,\p}}\left(\frac{1}{\p'(t)}\frac{d}{dt}\right)^nx(t),$$
while the right Caputo fractional derivative of $x$  is given by
$$\RD x(t):=   {I_{b-}^{n-\a,\p}}\left(-\frac{1}{\p'(t)}\frac{d}{dt}\right)^nx(t),$$
where
$$n=[\a]+1 \, \mbox{ for } \, \a\notin\mathbb N, \quad n=\a\, \mbox{ for } \, \a\in\mathbb N.$$
\end{definition}
It results that, if $\alpha=m$ is an integer, then
$$\LD x(t)=\left(\frac{1}{\p'(t)}\frac{d}{dt}\right)^mx(t) \quad \mbox{and} \quad \RD x(t)=\left(-\frac{1}{\p'(t)}\frac{d}{dt}\right)^mx(t).$$
On the other hand, if $\a\in\mathbb R^+\setminus\mathbb N$, then
$$\LD x(t)=\frac{1}{\Gamma(n-\a)}\int_a^t \p'(\t)(\p(t)-\p(\t))^{n-\a-1}\left(\frac{1}{\p'(\t)}\frac{d}{d\t}\right)^nx(\t)\,d\t$$
and
$$\RD x(t)=\frac{1}{\Gamma(n-\a)}\int_t^b \p'(\t)(\p(\t)-\p(t))^{n-\a-1} \left(-\frac{1}{\p'(\t)}\frac{d}{d\t}\right)^nx(\t)\,d\t.$$
Since we are interested in a generalization of the ordinary derivatives, we will consider this second case only.
For example, we have the following formulas. Given $n<\beta\in\mathbb R$,
$$\LD (\p(t)-\p(a))^{\beta-1}=\frac{\Gamma(\beta)}{\Gamma(\beta-\a)}(\p(t)-\p(a))^{\beta-\a-1}$$
and
$$\RD (\p(b)-\p(t))^{\beta-1}=\frac{\Gamma(\beta)}{\Gamma(\beta-\a)}(\p(b)-\p(t))^{\beta-\a-1}.$$
Also, given $\lambda\in\mathbb R$, we have
$$\LD E_\a(\lambda(\p(t)-\p(a))^\a)=\lambda E_\a(\lambda(\p(t)-\p(a))^\a)$$
 and
 $$\RD E_\a(\lambda (\p(b)-\p(t))^\a)=\lambda E_\a(\lambda (\p(b)-\p(t))^\a).$$
There exists a relation between the fractional integral and the fractional derivative operators. In a certain sense, they are the inverse
operation of each other. In fact, we have that
$$\LI \LD x(t)=x(t)-\sum_{k=0}^{n-1}\frac{ \left(\frac{1}{\p'(t)}\frac{d}{dt}\right)^kx(a)}{k!}(\p(t)-\p(a))^k$$
and
$$\RI \RD x(t)=x(t)-\sum_{k=0}^{n-1}\frac{ \left(-\frac{1}{\p'(t)}\frac{d}{dt}\right)^kx(b)}{k!}(\p(b)-\p(t))^k.$$
For the converse, the relation is the following:
$$\LD \LI x(t)=x(t)\quad \mbox{and} \quad \RD \RI x(t)=x(t).$$
One crucial formula, when dealing with variational problems, is a form of integration by parts, with respect to the fractional operators.
For these Caputo-type fractional derivatives, they are as follows.

\begin{theorem}\label{integrationparts} Given $x \in C[a,b]$ and $y\in C^n[a,b]$, the following holds:
\begin{align*}\int_a^b x(t)  \LD y(t)\,dt &= \int_a^b \RRD  \left(\frac{x(t)}{\p'(t)}\right) y(t)\p'(t)\,dt\\
&\quad+\left[\sum_{k=0}^{n-1}\left(-\frac{1}{\p'(t)}\frac{d}{dt}\right)^k {I_{b-}^{n-\a,\p}}\left(\frac{x(t)}{\p'(t)}\right)\cdot
\left(\frac{1}{\p'(t)}\frac{d}{dt}\right)^{n-k-1}y(t)\right]_{t=a}^{t=b}\end{align*}
and
\begin{align*}\int_a^b x(t)  \RD y(t)\,dt &= \int_a^b \RLD \left(\frac{x(t)}{\p'(t)}\right) y(t)\p'(t)\,dt\\
&\quad +\left[\sum_{k=0}^{n-1}(-1)^{n-k}\left(\frac{1}{\p'(t)}\frac{d}{dt}\right)^k {I_{a+}^{n-\a,\p}} \left(\frac{x(t)}{\p'(t)}\right) \cdot
\left(\frac{1}{\p'(t)}\frac{d}{dt}\right)^{n-k-1}y(t)\right]_{t=a}^{t=b}.\end{align*}
\end{theorem}

The paper is organized in the following way. In Section \ref{sec:main}, we present the main problem, and in Theorem \ref{teo:main} we prove an Euler--Lagrange type equation. In Section \ref{extension} we extend this result, by considering functionals where the lower bound of integration is greater than the lower bound of the fractional derivative. Next, in Section \ref{sec:iso}, we consider the variational problem subject to an integral constraint, in what is known as an isoperimetric problem, and in Section \ref{sec:legendre}, we deduce a second-order necessary condition that allows us to verify if the extremals are minimizers or not. In Sections \ref{sec:infinite} and \ref{sec:delay}, we consider the infinite horizon problem and the case where the Lagrange function has a delay, respectively. In Section \ref{sec:high}, we consider high-order derivatives in the functional, and derive the respective high-order Euler--Lagrange equation, and in Section \ref{sec:optimal}, we find a necessary condition that allows us to find the best fractional order to provide a minimum to the functional. Finally, in Section \ref{sec:suff}, we prove a sufficient condition  that guarantees the solutions of the Euler--Lagrange equations are almost  minimizers.

\section{Main results}

In this section, we study several variational problems, where the dynamic of the trajectories is described by a Caputo type fractional derivative.
We consider the initial point to be fixed, $x(a)=x_a$ ($x_a\in\mathbb R$), and the terminal point $T>a$ to be free, and thus it is also a variable of the problem.
We are interested in finding the optimal pair $(x,T)$ for the objective functionals.

\subsection{Fundamental problem}\label{sec:main}

The most important result in the calculus of variations is the so called Euler--Lagrange equation, which is a first order necessary
condition every extremizer of the functional must satisfy. For functionals depending on fractional operators, we find in the literature
numerous works already done for different kinds of fractional derivatives and initial/terminal conditions. Some examples are for the Riemann--Liouville
derivative \cite{AGRA1,Atanackovic,Baleanu2}, for the Caputo derivative \cite{AGRA2,Agnieszka1,withTatiana:Basia}, for the Riesz derivative
\cite{AGRA3,Almeida1,Baleanu1}.
We mention the recent books \cite{book2,book1}, where analytical and numerical methods are explained, respectively.

Our fractional variation problem with free terminal time is described in the following way. Let $L:[a,b]\times\mathbb R^2\to\mathbb R$
be a continuous function, such that there exist and are continuous the functions $\partial_2 L$ and $\partial_3 L$.
Define the functional
\begin{equation}\label{funct1}
J(x,T):=\int_a^T L(t,x(t),\LD x(t))\,dt, \quad (x,T)\in\Omega\times[a,b],
\end{equation}
where $\Omega$ is the set
$$\Omega:=\{x\in C^1[a,b]: x(a)\, \mbox{ is a fixed real}\},$$
which we endow with the norm
$$\|x\|_\Omega:=\max_{t\in[a,b]}|x(t)|+\max_{t\in[a,b]}\left|\LD x(t)\right|.$$
We say that $J$ assumes its minimum value at $(x^*,T^*)$ in $\Omega\times[a,b]$, relative to the norm
$$\|(x,t)\|:=\|x\|_\Omega+|t|,$$
provided that
$$\exists \epsilon>0 \, \forall (x,t)\in \Omega\times[a,b] \,: \, \|(x^*,T^*)-(x,t)\|<\epsilon \Rightarrow J(x^*,T^*) \leq J(x,t).$$
In this case, we say that  $(x^*,T^*)$ is a local minimum for $J$. An admissible variation for  $(x^*,T^*)$ is a pair
$(x^*+\epsilon v,T^*+\epsilon \triangle T)$, where $v\in C^1[a,b]$ and $v(a)=0$, $\epsilon,\triangle T \in\mathbb R$ and $|\epsilon| \ll 1$.
The next result provides a necessary condition that every local minimum for $J$ must  satisfy. In order to simplify the notation, we define $[x]$ as
$$[x](t):=(t,x(t),\LD x(t)), \quad (x,t)\in \Omega\times[a,b].$$

\begin{theorem}\label{teo:main} Suppose that $(x^*,T^*)$ is a local minimum for $J$ as in \eqref{funct1} on the space $\Omega\times[a,b]$.
If there exists and is continuous the function $t\mapsto {D_{T^*-}^{\a,\p}}\left(\partial_3 L[x^*](t)/\p'(t)\right)$ on $[a,T^*]$, then
\begin{equation}\label{ELeq}\partial_2 L[x^*](t)+{D_{T^*-}^{\a,\p}}\left(\frac{\partial_3 L[x^*](t)}{\p'(t)}\right)\p'(t)=0\end{equation}
for each $t\in[a,T^*]$, and at $t=T^*$, the following transversality conditions are satisfied:
\begin{equation}\label{TransCon}\left\{
\begin{array}{l}
{I_{T^*-}^{1-\a,\p}}\left(\frac{\partial_3 L[x^*](t)}{\p'(t)}\right)=0\\
L[x^*](t)=0.
\end{array}
\right.\end{equation}
\end{theorem}
\begin{proof} Consider an admissible variation of the optimal solution of the form  $(x^*+\epsilon v,T^*+\epsilon \triangle T)$.
If we define the function $j$ in a neighborhood of zero by the expression
$$j(\epsilon):=J(x^*+\epsilon v,T^*+\epsilon \triangle T),$$
we have that $j'(0)=0$. Differentiating $j$ at $\epsilon=0$, and using the integration by parts formula as in Theorem \ref{integrationparts}, we obtain
\begin{align*}
j'(0)&= \int_a^{T^*}\left[\partial_2 L[x^*](t) \cdot v(t)+\partial_3 L[x^*](t)  \cdot \LD v(t)\right]dt+\triangle T\cdot L[x^*](T^*)\\
&= \int_a^{T^*}\left[\partial_2 L[x^*](t)+{D_{T^*-}^{\a,\p}}\left(\frac{\partial_3 L[x^*](t)}{\p'(t)}\right)\p'(t)\right]v(t)\,dt\\
& \quad + \left[{I_{T^*-}^{1-\a,\p}}\left(\frac{\partial_3 L[x^*](t)}{\p'(t)}\right)\cdot v(t)\right]_{t=a}^{t=T^*} +\triangle T\cdot L[x^*](T^*).\\
\end{align*}
Since $v(a)=0$, if we consider $v(T^*)=0=\triangle T$, by the fundamental lemma of the calculus of variations (cf. \cite[Lemma 2.2.2]{Brunt}),
we conclude that for all $t\in[a,T^*]$, $x^*$ satisfies the condition
$$\partial_2 L[x^*](t)+{D_{T^*-}^{\a,\p}}\left(\frac{\partial_3 L[x^*](t)}{\p'(t)}\right)\p'(t)=0.$$
Therefore, we have
$$\left[{I_{T^*-}^{1-\a,\p}}\left(\frac{\partial_3 L[x^*](t)}{\p'(t)}\right)\cdot v(t)\right]_{t=a}^{t=T^*} +\triangle T \cdot L[x^*](T^*).$$
Since $v(T^*)$ and $\triangle T$ are free, we obtain the two transversality conditions.
\end{proof}

Equations like \eqref{ELeq} are called Euler--Lagrange equations, and they provide a first-order necessary condition that all
minimizers of the problem must satisfy. Notice that, although the functional  \eqref{funct1} depends on a Caputo type fractional derivative,
the Euler--Lagrange equation \eqref{ELeq} involves a Riemann--Liouville fractional derivative. We can rewrite it in such a way that the fractional equation depends
on the Caputo derivative as well. Observe that, given a differentiable function $f$ and $\a\in(0,1)$, we have
\begin{align*}
{^CD_{T^*-}^{\a,\p}}f(t)& = {D_{T^*-}^{\a,\p}}(f(t)-f(T^*))\\
& = {D_{T^*-}^{\a,\p}}f(t)-f(T^*)\cdot{D_{T^*-}^{\a,\p}}1\\
& = {D_{T^*-}^{\a,\p}}f(t)-\frac{f(T^*)}{\Gamma(1-\a)}(\p(T^*)-\p(t))^{-\a}.
\end{align*}
Using this new relation, the Euler--Lagrange equation is written in the form
$$\partial_2 L[x^*](t)+{^CD_{T^*-}^{\a,\p}}\left(\frac{\partial_3 L[x^*](t)}{\p'(t)}\right)\p'(t)+
\frac{\partial_3 L[x^*](T^*)}{\p'(T^*)\Gamma(1-\a)}(\p(T^*)-\p(t))^{-\a}\p'(t)=0.$$
The variational problem involving several dependent variables is similar, and we omit the proof here.

\begin{theorem} Consider the functional
$$J(x_1,\ldots,x_m,T):=\int_a^T L[x](t)\,dt,$$
where $m\in\mathbb N$, the functions $x_i$ verify the two assumptions $x_i\in C^1[a,b]$ and $x_i(a)$ is a fixed real,  for all $i=1,\ldots,m$,
 the real $T$ belongs to the interval $[a,b]$, and
$$[x](t):=(t,x_1(t),\ldots,x_m(t),{^CD_{a+}^{\a_1,\p}} x_1(t),\ldots,{^CD_{a+}^{\a_m,\p}} x_m(t)), \quad \a_1,\ldots,\a_m\in(0,1).$$
 Suppose that $(x_1^*,\ldots,x_m^*,T^*)$ is a local minimum for $J$, and that there exist and
are continuous the functions $t\mapsto {D_{T^*-}^{\a_i,\p}}\left(\partial_{i+1+m} L[x^*](t)/\p'(t)\right)$ on $[a,T^*]$, for all $i=1,\ldots,m$. Then,
$$\partial_{i+1} L[x^*](t)+{D_{T^*-}^{\a_i,\p}}\left(\frac{\partial_{i+1+m} L[x^*](t)}{\p'(t)}\right)\p'(t)=0$$
for all $i=1,\ldots,m$ and for all $t\in[a,T^*]$, and at $t=T^*$, the following holds:
$$\left\{
\begin{array}{l}
{I_{T^*-}^{1-\a_i,\p}}\left(\frac{\partial_{i+1+m} L[x^*](t)}{\p'(t)}\right)=0, \quad i=1,\ldots,m\\
L[x^*](t)=0.
\end{array}
\right.$$
\end{theorem}

\subsection{An extension}\label{extension}

In the previous problem, the lower limits of the cost functional and of the fractional derivative were the same, at $t=a$. In this section we generalize it,
by considering a cost functional starting at a point $t=A>a$.

\begin{theorem} Consider the functional
$$J(x,T):=\int_A^T L[x](t)\,dt,$$
where $x\in C^1[a,b]$ ($x(a)$ and $x(A)$ may be fixed or not) and $A,T\in[a,b]$ with $T>A$. Assume that $(x^*,T^*)$ is a local minimum for $J$, and that there
 exist and are continuous the functions $t\mapsto {D_{T^*-}^{\a,\p}}\left(\partial_3 L[x^*](t)/\p'(t)\right)$ on $[a,T^*]$ and
 $t\mapsto {D_{A-}^{\a,\p}}\left(\partial_3 L[x^*](t)/\p'(t)\right)$ on $[a,A]$. Then,
 $${D_{T^*-}^{\a,\p}}\left(\frac{\partial_3 L[x^*](t)}{\p'(t)}\right)-{D_{A-}^{\a,\p}}\left(\frac{\partial_3 L[x^*](t)}{\p'(t)}\right)=0,$$
for each $t\in[a,A]$, and
$$\partial_2 L[x^*](t)+{D_{T^*-}^{\a,\p}}\left(\frac{\partial_3 L[x^*](t)}{\p'(t)}\right)\p'(t)=0$$
for each $t\in[A,T^*]$. At $t=T^*$, the following holds:
$$\left\{
\begin{array}{l}
{I_{T^*-}^{1-\a,\p}}\left(\frac{\partial_3 L[x^*](t)}{\p'(t)}\right)=0\\
L[x^*](t)=0.
\end{array}
\right.$$
Moreover, if $x(a)$ is free, then at $t=a$:
$${I_{T^*-}^{1-\a,\p}}\left(\frac{\partial_3 L[x^*](t)}{\p'(t)}\right)-{I_{A-}^{1-\a,\p}}\left(\frac{\partial_3 L[x^*](t)}{\p'(t)}\right)=0,$$
and if $x(A)$ is free, then at $t=A$:
$${I_{A-}^{1-\a,\p}}\left(\frac{\partial_3 L[x^*](t)}{\p'(t)}\right)=0.$$
\end{theorem}
\begin{proof} The first variation of the functional at an extremum must vanish, and so we conclude that
\begin{align*}
0&= \int_A^{T^*}\left[\partial_2 L[x^*](t) \cdot v(t)+\partial_3 L[x^*](t)  \cdot \LD v(t)\right]dt+\triangle T\cdot L[x^*](T^*)\\
&= \int_a^{T^*}\left[\partial_2 L[x^*](t) \cdot v(t)+\partial_3 L[x^*](t)  \cdot \LD v(t)\right]dt\\
& \quad -\int_a^A\left[\partial_2 L[x^*](t) \cdot v(t)+\partial_3 L[x^*](t)  \cdot \LD v(t)\right]dt+\triangle T\cdot L[x^*](T^*)\\
&= \int_a^{T^*}\left[\partial_2 L[x^*](t)+{D_{T^*-}^{\a,\p}}\left(\frac{\partial_3 L[x^*](t)}{\p'(t)}\right)\p'(t)\right]v(t)\,dt
 + \left[{I_{T^*-}^{1-\a,\p}}\left(\frac{\partial_3 L[x^*](t)}{\p'(t)}\right)\cdot v(t)\right]_{t=a}^{t=T^*}\\
&\quad -\int_a^A\left[\partial_2 L[x^*](t)+{D_{A-}^{\a,\p}}\left(\frac{\partial_3 L[x^*](t)}{\p'(t)}\right)\p'(t)\right]v(t)\,dt
 - \left[{I_{A-}^{1-\a,\p}}\left(\frac{\partial_3 L[x^*](t)}{\p'(t)}\right)\cdot v(t)\right]_{t=a}^{t=A}\\
& \quad +\triangle T\cdot L[x^*](T^*)\\
&=\int_a^A\left[{D_{T^*-}^{\a,\p}}\left(\frac{\partial_3 L[x^*](t)}{\p'(t)}\right)-{D_{A-}^{\a,\p}}\left(\frac{\partial_3 L[x^*](t)}{\p'(t)}\right)\right]
\p'(t)\cdot v(t)\,dt \\
& \quad +\int_A^{T^*}\left[\partial_2 L[x^*](t)+{D_{T^*-}^{\a,\p}}\left(\frac{\partial_3 L[x^*](t)}{\p'(t)}\right)\p'(t)\right]\cdot v(t)\,dt\\
&\quad + \left[{I_{T^*-}^{1-\a,\p}}\left(\frac{\partial_3 L[x^*](t)}{\p'(t)}\right)\cdot v(t)\right]_{t=a}^{t=T^*}
- \left[{I_{A-}^{1-\a,\p}}\left(\frac{\partial_3 L[x^*](t)}{\p'(t)}\right)\cdot v(t)\right]_{t=a}^{t=A}\\
& \quad +\triangle T\cdot L[x^*](T^*).
\end{align*}
By the arbitrariness of $v$ and $\triangle T$, we obtain the necessary optimality conditions.
\end{proof}

\subsection{Isoperimetric problem}\label{sec:iso}

We formulate now the variational problem when in presence of an integral constraint. We refer to \cite{Almeida2,Odziehjewicz}, where similar
problems were solved involving fractional derivatives. This kind of problems are known in the literature as isoperimetric problems. The most
ancient problem of this type goes back to the Ancient Greece, with the question of finding out which of all closed planar curves of the same length would enclose the greatest area. Nowadays, an isoperimetric problem is a variational problem, restricted to a subclass of functions satisfying
a side condition of the form
$$\int_a^b M(t,x(t),x'(t))\,dt=\mbox{constant.}$$
Here, we replace the ordinary derivative by a fractional derivative, and since the terminal time is free, the integral value is not a constant,
but a function depending on the terminal time. Let $M:[a,b]\times\mathbb R^2\to\mathbb R$ be a continuous function, such that there exist and are
continuous the functions $\partial_2 M$ and $\partial_3 M$.

\begin{theorem}\label{teo:isop} Suppose that $(x^*,T^*)$ is a local minimum for $J$ as in \eqref{funct1} on the space $\Omega\times[a,b]$,
subject to the integral constraint
$$G(x,T):=\int_a^TM[x](t)\,dt=\Phi(T),$$
where $\Phi\in C^1[a,b]$. Suppose that $(x^*,T^*)$ is not a solution of the equation
\begin{equation}\label{aux1}\partial_2 M[x](t)+{D_{T^*-}^{\a,\p}}\left(\frac{\partial_3 M[x](t)}{\p'(t)}\right)\p'(t)=0,
\quad \forall t \in[a,T^*],\end{equation}
and that there exist and are continuous the functions $t\mapsto {D_{T^*-}^{\a,\p}}\left(\partial_{3} L[x^*](t)/\p'(t)\right)$ and $t\mapsto {D_{T^*-}^{\a,\p}}
\left(\partial_{3} M[x^*](t)/\p'(t)\right)$ on $[a,T^*]$.
Then, there exists a real constant $\lambda$, such that if we define the augmented function $F:=L+\lambda M$, then $(x^*,T^*)$ satisfies the equation
$$\partial_2 F[x](t)+{D_{T^*-}^{\a,\p}}\left(\frac{\partial_3 F[x](t)}{\p'(t)}\right)\p'(t)=0, \quad \forall t \in[a,T^*],$$
and the system
$$\left\{
\begin{array}{l}
{I_{T^*-}^{1-\a,\p}}\left(\frac{\partial_3 F[x](t)}{\p'(t)}\right)=0\\
F[x](t)=\lambda \Phi'(t)
\end{array}\right. \quad \mbox{at} \, \, t=T^*.$$
\end{theorem}
\begin{proof} Consider admissible variations of two parameters of kind  $(x^*+\epsilon_1 v_1+\epsilon_2v_2,T^*+\epsilon_2 \triangle T)$, where
$v_1,v_2\in C^1[a,b]$ with $v_1(a)=v_2(a)=0$, and $\epsilon_1,\epsilon_2,\triangle T\in\mathbb R$ with $|\epsilon_1|,|\epsilon_2|\ll1$. Define the two functions:
\begin{align*}j(\epsilon_1,\epsilon_2):=&J(x^*+\epsilon_1 v_1+\epsilon_2v_2,T^*+\epsilon_2 \triangle T)\\
g(\epsilon_1,\epsilon_2):=&G(x^*+\epsilon_1 v_1+\epsilon_2v_2,T^*+\epsilon_2 \triangle T)-\Phi(T^*+\epsilon_2 \triangle T).
\end{align*}
Since
\begin{align*}\frac{\partial g}{\partial \epsilon_1}(0,0)&=\int_a^{T^*}\left[\partial_2 M[x^*](t)+{D_{T^*-}^{\a,\p}}
\left(\frac{\partial_3 M[x^*](t)}{\p'(t)}\right)\p'(t)\right]v_1(t)\,dt\\
& \quad+ \left[{I_{T^*-}^{1-\a,\p}}\left(\frac{\partial_3 M[x^*](t)}{\p'(t)}\right)\cdot v_1(t)\right]_{t=a}^{t=T^*} ,\end{align*}
and $(x^*,T^*)$ is not a solution for Eq. \eqref{aux1}, we deduce that there exists a function $v_1\in C^1[a,b]$ such that
 $\partial g/\partial \epsilon_1(0,0)\not=0$. We can appeal to the implicit function theorem, which asserts that there exists a function
 $\epsilon_1(\cdot)$, defined on a neighborhood of zero, such that $g(\epsilon_1(\epsilon_2),\epsilon_2)=0$. Thus, there exists a subfamily
 of admissible variations satisfying the integral constraint. Attending that $j$ is minimum at $(0,0)$ subject to the constraint $g(\cdot,\cdot)=0$,
 and since $\nabla g (0,0)\not=(0,0)$, by the Lagrange multiplier rule, there exists a real number $\lambda$ such that
$$\nabla(j+\lambda g)(0,0)=(0,0).$$
In particular, $\partial (j+\lambda g)/\partial \epsilon_2(0,0)=0$. Repeating the calculations as done before, we arrive at the desired formulas.
\end{proof}

\subsection{Legendre condition}\label{sec:legendre}

We  now formulate a second-order necessary condition, usually called Legendre condition, which provides us with a necessary condition for minimization.
In \cite{Lazo}, by the first time, a Legendre type condition was obtained for fractional variational calculus. Here, we derive a similar condition to
a more general form of fractional derivative. Assume now that the Lagrange function $L$ is such that its second order partial derivatives
$\partial^2_{ij}L$, with $i,j\in\{2,3\}$, exist and are continuous.

\begin{theorem} Suppose that $(x^*,T^*)$ is a local minimum for $J$ as in \eqref{funct1} on the space $\Omega\times[a,b]$. Then for all $ t \in[a,T^*]$,
$$\partial^2_{33}L[x^*](t)\geq0. $$
\end{theorem}
\begin{proof} Let us consider variations over $x^*$ only, that is, we restrict to the case $\triangle T=0$. So, if we consider
$j(\epsilon):=J(x^*+\epsilon v,T^*)$, we have $j''(0)\geq0$, and so we conclude that
\begin{equation}\label{aux2}
\int_a^{T^*}\left[\partial^2_{22}L[x^*](t)\cdot v^2(t)+2\,\partial^2_{23}L[x^*](t)\cdot v(t)\LD v(t)+\partial^2_{33}L[x^*](t)\cdot
\left(\LD v(t)\right)^2\right]\,dt\geq0.
\end{equation}
Suppose that the Legendre condition is violated at some some $t_0\in[a,T^*]$:
$$\partial^2_{33}L[x^*](t_0)<0.$$
Then, there exists a subinterval $[c,d]\subseteq [a,T^*]$ and three real constants $C_1,C_2,C_3$ with $C_3<0$ such that
$$\partial^2_{22}L[x^*](t)\leq C_1, \quad \partial^2_{23}L[x^*](t)\leq C_2, \quad \partial^2_{33}L[x^*](t)\leq C_3, $$
for all $t \in[c,d]$. Define the function $h:[c,d]\to\mathbb R$ by the formula
\begin{align*}h(t)&:= (\a+2)(\p(t)-\p(c))^{\a+1}-2\frac{\a+4}{\p(d)-\p(c)}(\p(t)-\p(c))^{\a+2}\\
    & \quad +\frac{\a+10}{(\p(d)-\p(c))^2}(\p(t)-\p(c))^{\a+3}-\frac{4}{(\p(d)-\p(c))^3}(\p(t)-\p(c))^{\a+4}.
\end{align*}
Then, $h(c)=0=h(d)$ and $h'(c)=0=h'(d)$.
Once
\begin{align*}{^CD_{c+}^{\a,\p}} h(t)&= \Gamma(\a+3)(\p(t)-\p(c))-\frac{(\a+4)\Gamma(\a+3)}{\p(d)-\p(c)}(\p(t)-\p(c))^{2}\\
    & \quad +\frac{(\a+10)\Gamma(\a+4)}{6(\p(d)-\p(c))^2}(\p(t)-\p(c))^{3}-\frac{\Gamma(\a+5)}{6(\p(d)-\p(c))^3}(\p(t)-\p(c))^{4},
\end{align*}
we also have ${^CD_{c+}^{\a,\p}} h(c)=0={^CD_{c+}^{\a,\p}} h(d)$. Besides this, for every $t\in[c,d]$,
\begin{align*}h(t)&\leq  (\a+2)(\p(t)-\p(c))^{\a+1}+\frac{\a+10}{(\p(d)-\p(c))^2}(\p(t)-\p(c))^{\a+3}\\
 & \leq (2\a+12) (\p(d)-\p(c))^{\a+1}\leq 14  (\p(d)-\p(c))^{\a+1}
\end{align*}
and
\begin{align*}{^CD_{c+}^{\a,\p}} h(t)&\leq \Gamma(\a+3)(\p(t)-\p(c))+\frac{(\a+10)\Gamma(\a+4)}{6(\p(d)-\p(c))^2}(\p(t)-\p(c))^{3}\\
  & \leq \frac{\Gamma(\a+3)(\a^2+13\a+36)}{6}(\p(d)-\p(c))\leq 50(\p(d)-\p(c)).
\end{align*}
Define the function $v:[a,b]\to\mathbb R$ by the rule
$$v(t):=\left\{\begin{array}{ll}
h(t), & \quad \mbox{if } t \in [c,d]\\
0, & \quad \mbox{if } t \notin [c,d].\\
\end{array}\right.$$
By the properties of function $h$, we have that $v\in C^1[a,b]$, $v(a)=0$ and
$$ \LD v(t):=\left\{\begin{array}{ll}
{^CD_{c+}^{\a,\p}}v(t), & \quad \mbox{if } t \in [c,d]\\
0, & \quad \mbox{if } t \notin[c,d].\\
\end{array}\right.$$
Note that, for $t>d$, $\LD v(t)= {^CD_{c+}^{\a,\p}}h(d)=0$. Replacing this variation into Eq. \eqref{aux2}, we get
\begin{align*}0&\leq\int_a^{T^*}\left[\partial^2_{22}L[x^*](t)\cdot v^2(t)+2\,\partial^2_{23}L[x^*](t)\cdot v(t)\LD v(t)+\partial^2_{33}L[x^*](t)\cdot(\LD v(t))^2\right]\,dt\\
&\leq \int_a^{T^*}\left[ 14^2 C_1(\p(d)-\p(c))^{2\a+2}+2\cdot 14\cdot 50 C_2(\p(d)-\p(c))^{\a+2}+50^2C_3 (\p(d)-\p(c))^2 \right]\,dt\\
&=(\p(d)-\p(c))^2 (T^*-a)\left[ 196 C_1 (\p(d)-\p(c))^{2\a}+1400C_2(\p(d)-\p(c))^\a+2500C_3 \right]<0\end{align*}
if we assume that $|d-c|\ll1$, and thus we obtain a contradiction.
\end{proof}

\subsection{Infinite horizon problem}\label{sec:infinite}

We study now a new problem, important when we want to consider the effects at a long term.  This issue is especially pertinent when the
 planning horizon is assumed to be of infinite length.
The objective functional is given by an  improper integral, the initial state $x(a)$ is fixed and the terminal state (at infinity) is free, that is,
no constraints are imposed on the behaviour of the admissible trajectories at large times.
 This kind of problems are known as infinite horizon problems, where the objective functional is given by
\begin{equation}\label{funct2}
J(x):=\int_a^\infty L(t,x(t),\LD x(t))\,dt, \quad x\in\Omega_\infty,
\end{equation}
where $\Omega_\infty$ is the set
$$\Omega_\infty:=\{x\in C^1[a,\infty): x(a)\, \mbox{ is a fixed real}\},$$
endowed with the norm
$$\|x\|_{\Omega_\infty}:=\sup_{t\in[a,\infty)}|x(t)|+\sup_{t\in[a,\infty)}\left|\LD x(t)\right|.$$
We have to be careful when defining a minimal curve for functional \eqref{funct2}, since any admissible function for which the improper integral
 diverges to $-\infty$ would be a minimal path, according to the usual definition of minimum.
Here, we follow the one presented in \cite{brock}.
A curve $x^*$ in $\Omega_\infty$ is a local weakly minimal for $J$ as in \eqref{funct2} if there exists some $ \epsilon>0$ such that, for all
 $x\in \Omega_\infty$, if $\|x^*-x\|_{\Omega_\infty}<\epsilon$, then the lower limit
$$\lim_{T\to\infty}\inf_{T^* \geq T}\int_{a}^{T^*}[L[x^*](t)-L[x](t)] \, dt \leq 0.$$
For the following result, we will need some extra functions. Fixed two functions $x^*,v \in C^1[a,\infty)$, and given $|\epsilon|\ll 1$ and
$T^*\geq a$, define
\begin{align*}
 A(\epsilon, T^*)& := \int_a^{T^*} \frac{L[x^*+\epsilon v](t)- L[x^*](t)}{\epsilon}dt;\\
V(\epsilon, T)& := \inf_{T^* \geq T}\int_a^{T^*} [L[x^*+\epsilon v](t)-L[x^*](t)]dt;\\
W(\epsilon)&:= \lim_{T\to\infty} V(\epsilon,T).
\end{align*}

\begin{theorem} Let $x^*$ be a local weakly minimal for $J$ as in \eqref{funct2}. Suppose that:
\begin{enumerate}
\item $\displaystyle \lim_{\epsilon \to 0} \frac{V(\epsilon, T) }{\epsilon}$ exists for all $T$;
\item $\displaystyle \lim_{T\to\infty}\frac{V(\epsilon, T) }{\epsilon}$ exists uniformly for all $\epsilon$;
\item For every $T^*> a$ and $\epsilon\not=0$, there exists a sequence $\left(A(\epsilon, T^*_n)\right)_{n \in \mathbb{N}}$
such that
$$\lim_{n \to \infty} A(\epsilon, T^*_n)= \inf_{T^*\geq T} A(\epsilon, T^*)$$
uniformly for $\epsilon$.
\end{enumerate}
If there exists and is continuous the function $t\mapsto {D_{T^*-}^{\a,\p}}\left(\partial_3 L[x^*](t)/\p'(t)\right)$ on $[a,T^*]$, for all $T^*\geq t\geq a$,
 then
$$\partial_2 L[x^*](t)+{D_{T^*-}^{\a,\p}}\left(\frac{\partial_3 L[x^*](t)}{\p'(t)}\right)\p'(t)=0,$$
for all $T^*\geq t\geq a$. Also, we have
$$\lim_{T\to\infty}\inf_{T^* \geq T}{I_{T^*-}^{1-\a,\p}}\left(\frac{\partial_3 L[x^*](t)}{\p'(t)}\right)=0 \quad \mbox{at} \, \, t=T^*.$$
\end{theorem}

\begin{proof} By the definition of minimum curve for the infinite horizon problem, we have that $W(\epsilon)\geq0$ in a neighborhood of zero,
and $W(0)=0$. Thus, $W'(0)=0$ and so we have the following:
\begin{align*}
0 &=\lim_{\epsilon\to 0} \frac{W(\epsilon)}{\epsilon}= \lim_{\epsilon\to 0}\lim_{T\to \infty}\frac{V(\epsilon, T) }{\epsilon}\\
&= \lim_{T\to\infty}\lim_{\epsilon\to 0} \frac{V(\epsilon, T) }{\epsilon}=\lim_{T\to\infty}\lim_{\epsilon\to 0}\inf_{T^* \geq T} A(\epsilon, T^*)\\
&= \lim_{T\to\infty}\lim_{\epsilon\to 0} \lim_{n\to \infty} A(\epsilon, T^*_n)=\lim_{T\to\infty}\lim_{n\to \infty} \lim_{\epsilon\to 0} A(\epsilon, T^*_n) \\
&= \lim_{T\to\infty}\inf_{T^*\geq T} \lim_{\epsilon\to 0} A(\epsilon, T^*)=\lim_{T\to\infty}\inf_{T^* \geq T}\lim_{\epsilon\to 0}\int_{a}^{T^*} \frac{L[ x^*
 + \epsilon v](t)- L[x^*](t)}{\epsilon} \, dt\\
&= \lim_{T\to\infty}  \inf_{T^* \geq T}\int_{a}^{T^*}\left[\partial_2 L[x^*](t) \cdot v(t)+\partial_3 L[x^*](t)  \cdot \LD v(t)\right]dt\\
&= \lim_{T\to\infty}  \inf_{T^* \geq T}\left[\int_a^{T^*}\left[\partial_2 L[x^*](t)+{D_{T^*-}^{\a,\p}}
\left(\frac{\partial_3 L[x^*](t)}{\p'(t)}\right)\p'(t)\right]v(t)\,dt\right.\\
& \quad \left.+\left[{I_{T^*-}^{1-\a,\p}}\left(\frac{\partial_3 L[x^*](t)}{\p'(t)}\right)\cdot v(t)\right]_{t=a}^{t=T^*}\right].
\end{align*}
If we assume that $v(T^*)=0$, we deduce that
$$\lim_{T\to\infty}  \inf_{T^* \geq T}\int_a^{T^*}\left[\partial_2 L[x^*](t)+{D_{T^*-}^{\a,\p}}\left(\frac{\partial_3 L[x^*](t)}{\p'(t)}\right)
\p'(t)\right]v(t)\,dt=0,$$
and so (see \cite{AlmeidaIn})
$$\partial_2 L[x^*](t)+{D_{T^*-}^{\a,\p}}\left(\frac{\partial_3 L[x^*](t)}{\p'(t)}\right)\p'(t)=0,$$
for all $t\geq a$ and for all $T^* \geq t$. Also, using this last condition, we get
$$\lim_{T\to\infty}\inf_{T^* \geq T}{I_{T^*-}^{1-\a,\p}}\left(\frac{\partial_3 L[x^*](t)}{\p'(t)}\right)=0 \quad \mbox{at} \, \, t=T^*.$$
\end{proof}

\subsection{Variational principles with delay}\label{sec:delay}

In this section we consider time-delay variational problems. This is an important subject, since in many systems there is almost always
a time delay \cite{Richard,Salamon}. A natural generalization of such theory is to replace ordinary derivatives by fractional derivatives,
since fractional operators contain memory, and their present state is determined by all past states. There exist
already some works dealing with fractional operators, for example \cite{AlmeidaD,BaleanuD1,BaleanuD2}.
Let $L:[a,b]\times\mathbb R^3\to\mathbb R$ be a continuous function such that there exist and are continuous the functions $\partial_i L$, for $i=2,3,4$.
Given $\t>0$ such that $\t<b-a$, define the functional
\begin{equation}\label{functDelay}
J(x,T):=\int_a^T L(t,x(t),x(t-\t),\LD x(t))\,dt, \quad (x,T)\in\Omega_\t\times[a,b],
\end{equation}
where
$$\Omega_\t:=\{x\in C^1[a-\t,b]: x(t)=\theta(t)\, \mbox{ for }\, t\in[a-\t,a]\},$$
and $\theta$ is a given function. Let $[x]_\t$ denote the vector
$$[x]_\t(t):=(t,x(t),x(t-\t),\LD x(t)).$$

\begin{theorem}\label{teo:Delay} Let the pair $(x^*,T^*)$ be local minimum for $J$ as in \eqref{functDelay}.
If there exist and are continuous the functions $t\mapsto {D_{(T^*-\t)-}^{\a,\p}}\left(\partial_4 L[x^*](t)/\p'(t)\right)$ and $t\mapsto {D_{T^*-}^{\a,\p}}\left(\partial_4 L[x^*](t)/\p'(t)\right)$
 on $[a,T^*]$, then for all $t\in[a,T^*-\t]$,
$$\partial_2 L[x^*]_\t(t)+\partial_3 L[x^*]_\t(t+\t)+{D_{(T^*-\t)-}^{\a,\p}}\left(\frac{\partial_4 L[x^*]_\t(t)}{\p'(t)}\right)\p'(t)$$
$$-\frac{1}{\Gamma(1-\a)}\left(\frac{1}{\p'(t)}\frac{d}{dt}\right)\int_{T^*-\t}^{T^*}\p'(s)(\p(s)-\p(t))^{-\a}\frac{\partial_4 L[x^*]_\t(s)}{\p'(s)}
\, ds \cdot \p'(t)=0,$$
and for all $t\in[T^*-\t,T^*]$,
$$\partial_2 L[x^*]_\t(t)+{D_{T^*-}^{\a,\p}}\left(\frac{\partial_4 L[x^*]_\t(t)}{\p'(t)}\right)\p'(t)=0.$$
Also, at $t=T^*$, it is true that
$$\left\{
\begin{array}{l}
{I_{T^*-}^{1-\a,\p}}\left(\frac{\partial_4 L[x^*]_\t(t)}{\p'(t)}\right)=0\\
L[x^*]_\t(t)=0.
\end{array}
\right.$$
\end{theorem}
\begin{proof} Let $v\in C^1[a-\t,b]$ be such that $v(t)=0$, for all $t\in[a-\t,a],$ and consider variations of the form
$(x^*+\epsilon v,T^*+\epsilon \triangle T)$. Since the first variation of the functional must vanish at an extremum point, we have
$$\int_a^{T^*}\left[\partial_2 L[x^*]_\t(t) \cdot v(t)+\partial_3 L[x^*]_\t(t) \cdot v(t-\t)+
\partial_4 L[x^*](t)  \cdot \LD v(t)\right]dt+\triangle T\cdot L[x^*](T^*)=0.$$
Observe that
$$\int_a^{T^*} \partial_3 L[x^*]_\t(t) \cdot v(t-\t)\,dt=\int_a^{T^*-\t} \partial_3 L[x^*]_\t(t+\t) \cdot v(t)\,dt$$
since $v(t)=0$, for all $t\in[a-\t,a]$.
Also, since
\begin{align*}{D_{T^*-}^{\a,\p}}\left(\frac{\partial_4 L[x^*]_\t(t)}{\p'(t)}\right)&
={D_{(T^*-\t)-}^{\a,\p}}\left(\frac{\partial_4 L[x^*]_\t(t)}{\p'(t)}\right)\\
& \quad-\frac{1}{\Gamma(1-\a)}\left(\frac{1}{\p'(t)}\frac{d}{dt}\right)\int_{T^*-\t}^{T^*}\p'(s)(\p(s)-\p(t))^{-\a}\frac{\partial_4 L[x^*]_\t(s)}{\p'(s)}\,ds
\end{align*}
we obtain the following
$$\int_a^{T^*} \partial_4 L[x^*]_\t(t) \cdot \LD v(t)\,dt=\int_a^{T^*} {D_{T^*-}^{\a,\p}}\left(\frac{\partial_4 L[x^*]_\t(t)}{\p'(t)}\right)\p'(t)
\cdot v(t)\,dt+\left[ {I_{T^*-}^{1-\a,\p}}\left(\frac{\partial_4 L[x^*]_\t(t)}{\p'(t)}\right)\cdot v(t)\right]_a^{T^*}$$
$$=\int_a^{T^*-\t}\left[{D_{(T^*-\t)-}^{\a,\p}}\left(\frac{\partial_4 L[x^*]_\t(t)}{\p'(t)}\right)\p'(t)
-\frac{1}{\Gamma(1-\a)}\left(\frac{1}{\p'(t)}\frac{d}{dt}\right)\int_{T^*-\t}^{T^*}\p'(s)(\p(s)-\p(t))^{-\a} \frac{\partial_4 L[x^*]_\t(s)}{\p'(s)}\,ds\right.$$
$$\cdot \p'(t)\Bigg]v(t)\,dt
+\int_{T^*-\t}^{T^*} {D_{T^*-}^{\a,\p}}\left(\frac{\partial_4 L[x^*]_\t(t)}{\p'(t)}\right)\p'(t)
\cdot v(t)\,dt+\left[ {I_{T^*-}^{1-\a,\p}}\left(\frac{\partial_4 L[x^*]_\t(t)}{\p'(t)}\right)\cdot v(t)\right]_a^{T^*}.$$
Finally, combining all the previous formulas, we obtain
$$\int_a^{T^*-\t}\left[ \partial_2 L[x^*]_\t(t)+\partial_3 L[x^*]_\t(t+\t)+{D_{(T^*-\t)-}^{\a,\p}}
\left(\frac{\partial_4 L[x^*]_\t(t)}{\p'(t)}\right)\p'(t)\right.$$
$$\left. -\frac{1}{\Gamma(1-\a)}\left(\frac{1}{\p'(t)}\frac{d}{dt}\right)\int_{T^*-\t}^{T^*}\p'(s)(\p(s)-\p(t))^{-\a}\frac{\partial_4 L[x^*]_\t(s)}{\p'(s)}
\, ds \cdot \p'(t) \right]v(t)dt$$
$$+\int_{T^*-\t}^{T^*}\left[ \partial_2 L[x^*]_\t(t)+{D_{T^*-}^{\a,\p}}\left(\frac{\partial_4 L[x^*]_\t(t)}{\p'(t)}\right)\p'(t)\right]v(t)dt$$
$$+\left[ {I_{T^*-}^{1-\a,\p}}\left(\frac{\partial_4 L[x^*]_\t(t)}{\p'(t)}\right)\cdot v(t)\right]_a^{T^*}+\triangle T\cdot L[x^*]_\t(T^*)=0.$$
Since $v$ is arbitrary on the interval $[a,T^*]$, as well as $\triangle T$, we obtain the desired result.
\end{proof}

\subsection{High order derivatives}\label{sec:high}

So far we considered a fractional order as a real between 0 and 1. Using similar techniques as the ones presented in the proof of Theorem
\ref{teo:main}, we can generalize the previous results in order to include high order derivatives. We show how to do it for the basic problem
 of the calculus of variations, and we deduce the respective Euler--Lagrange equation.

\begin{theorem} Consider the functional
$$J(x,T):=\int_a^T L(t,x(t),{^CD_{a+}^{\a_1,\p}} x(t),\ldots, {^CD_{a+}^{\a_m,\p}} x(t))\,dt, \quad (x,T)\in\Omega_m\times[a,b],$$
where
\begin{enumerate}
\item $m\in\mathbb N$, and for all $n\in\{1,\ldots,m\}$, we have $\a_n\in(n-1,n)$;
\item $L:[a,b]\times\mathbb R^{m+1}\to\mathbb R$ is a continuous function;
\item there exist and are continuous the functions $\partial_2 L$, $\partial_3 L$, $\ldots$, and  $\partial_{m+2} L$;
\item $\Omega_m:=\{x\in C^m[a,b]: x(a), x^{(1)}(a),\ldots, x^{(m-1)}(a) \, \mbox{ are fixed reals}\}$.
\end{enumerate}
Suppose that $(x^*,T^*)$ is a local minimum for $J$. If, for all $n\in\{1,\ldots,m\}$, there exist and are continuous the functions
$t\mapsto {D_{T^*-}^{\a_n,\p}}\left(\partial_{n+2} L[x^*](t)/\p'(t)\right)$ on $[a,T^*]$, then
$$\partial_2 L[x^*]_m(t)+\sum_{n=1}^m{D_{T^*-}^{\a_n,\p}}\left(\frac{\partial_{n+2} L[x^*]_m(t)}{\p'(t)}\right)\p'(t)=0,$$
for each $t\in[a,T^*]$, and at $t=T^*$, we have
$$\left\{
\begin{array}{l}
\sum_{n=k}^m \left(-\frac{1}{\p'(t)}\frac{d}{dt}\right)^{n-k}{I_{T^*-}^{n-\a_n,\p}}\left(\frac{\partial_{n+2} L[x^*]_m(t)}{\p'(t)}\right)=0,
\quad \, k=1,\ldots,m,\\
L[x^*]_m(t)=0,
\end{array}
\right.$$
where
$$[x^*]_m(t):=(t,x^*(t),{^CD_{a+}^{\a_1,\p}} x^*(t),\ldots, {^CD_{a+}^{\a_m,\p}} x^*(t)).$$
\end{theorem}
\begin{proof} Consider admissible variations of the form  $(x^*+\epsilon v,T^*+\epsilon \triangle T)$, where $v\in C^m[a,b]$ and $v(a)=0=v^{(n)}(a)$,
 for all $n=1,\ldots,m-1$. Since the first variation of the functional must vanish at an extremum point, we deduce the following:
$$\int_a^{T^*}\left[\partial_2 L[x^*]_m(t) \cdot v(t)+\sum_{n=1}^m\partial_{n+2} L[x^*]_m(t)\cdot{^CD_{a+}^{\a_n,\p}} v(t)\right]dt+\triangle T\cdot
 L[x^*]_m(T^*)=0.$$
Integrating by parts, we get
$$\int_a^{T^*}\left[ \partial_2 L[x^*]_m(t)+\sum_{n=1}^m{D_{T^*-}^{\a_n,\p}}\left(\frac{\partial_{n+2} L[x^*]_m(t)}{\p'(t)}\right)\p'(t)\right]v(t)\,dt+
\triangle T\cdot L[x^*]_m(T^*)$$
$$+\left[\sum_{n=1}^m \left(-\frac{1}{\p'(t)}\frac{d}{dt}\right)^{n-1}{I_{T^*-}^{n-\a_n,\p}}\left(\frac{\partial_{n+2} L[x^*]_m(t)}{\p'(t)}\right)
\cdot v(t)\right.$$
$$+\sum_{n=2}^m \left(-\frac{1}{\p'(t)}\frac{d}{dt}\right)^{n-2}{I_{T^*-}^{n-\a_n,\p}}\left(\frac{\partial_{n+2} L[x^*]_m(t)}{\p'(t)}\right)\cdot
\left(\frac{1}{\p'(t)}\frac{d}{dt}\right)v(t)$$
$$\left.+\ldots+{I_{T^*-}^{m-\a_m,\p}}\left(\frac{\partial_{m+2} L[x^*]_m(t)}{\p'(t)}\right)\cdot \left(\frac{1}{\p'(t)}\frac{d}{dt}\right)^{m-1}v(t)
\right]_{t=a}^{t=T^*}=0.$$
Choosing appropriate variations, we deduce the result.
\end{proof}

\subsection{Optimal fractional order}\label{sec:optimal}

One advantage of considering fractional derivatives in modelling phenomena is that they may describe more efficiently the dynamics than ordinary derivatives.
So, a natural question to pose is what should be the order of the fractional operator, in order to minimize the functional. The next result provides necessary
conditions to determine the fractional order.

\begin{theorem}\label{FracOrder} Consider the functional
$$J(x,T,\a):=\int_a^T L[x](t)\,dt,$$
where $x\in C^1[a,b]$ with $x(a)$ fixed, $T\in[a,b]$ and $\a\in(0,1)$. Assume that $(x^*,T^*,\a^*)$ is a local minimum for $J$, and that there
 exists and is continuous the function $t\mapsto {D_{T^*-}^{\a,\p}}\left(\partial_3 L[x^*](t)/\p'(t)\right)$ on $[a,T^*]$. Given $t\in[a,T^*]$,
 define a function $ \Lambda_t:(0,1)\to\mathbb R$ by the rule $\Lambda_t(\a):=\LD [x^*](t)$. Then,
$$\partial_2 L[x^*](t)+{D_{T^*-}^{\a,\p}}\left(\frac{\partial_3 L[x^*](t)}{\p'(t)}\right)\p'(t)=0$$
for each $t\in[a,T^*]$, and at $t=T^*$, the following holds:
$$\left\{
\begin{array}{l}
{I_{T^*-}^{1-\a,\p}}\left(\frac{\partial_3 L[x^*](t)}{\p'(t)}\right)=0\\
L[x^*](t)=0.
\end{array}
\right.$$
Also, we have the following equality:
$$\int_a^{T^*}\partial_3 L[x^*](t)\cdot \Lambda_t'(\a^*)\,dt=0.$$
\end{theorem}
\begin{proof} Define
$$j(\epsilon):=J(x^*+\epsilon v, T^*+\epsilon \triangle T, \a^*+\epsilon \triangle \a),$$
where $v\in C^1[a,b]$ with $v(a)=0$ and $\triangle T,\triangle \a$ are two fixed reals.
Since $j'(0)=0$, using integration by parts, we arrive at the formula
$$\int_a^{T^*}\left[\partial_2 L[x^*](t)+{D_{T^*-}^{\a,\p}}\left(\frac{\partial_3 L[x^*](t)}{\p'(t)}\right)\p'(t)\right]\cdot v(t)\,dt
+ \left[{I_{T^*-}^{1-\a,\p}}\left(\frac{\partial_3 L[x^*](t)}{\p'(t)}\right)\cdot v(t)\right]_{t=a}^{t=T^*}$$
$$+\triangle T\cdot L[x^*](T^*)+\triangle \a\cdot \int_a^{T^*}\partial_3 L[x^*](t)\cdot \Lambda_t'(\a^*)\,dt=0.$$
By the arbitrariness of $v$, $\triangle T$ and $\triangle \a$, we prove the desired result.
\end{proof}

\subsection{Sufficient conditions}\label{sec:suff}

So far we deduced necessary conditions that every minimizer of the functional must verify. Now, we will deal with a sufficient condition, and for that some conditions of convexity over the Lagrangian are needed.

\begin{theorem} Assume that $L$ is convex in $[a,b]\times\mathbb R^2$, in the sense that
$$L(t,x+v,y+w)-L(t,x,y)\geq \partial_2 L(t,x,y)v+\partial_3 L(t,x,y)w,$$
for all $t\in[a,b]$ and $x,y,v,w\in\mathbb R$. If  $(x^*,T^*)$ satisfies the Euler--Lagrange equation \eqref{ELeq} and the transversality
conditions \eqref{TransCon}, then
$$\forall\epsilon>0\,\exists\delta>0 \, : \, \|(v,\triangle T)\|<\delta \Rightarrow J(x^*+v,T^*+\triangle T)-J(x^*,T^*)\geq-\epsilon$$
on the space $\Omega\times[a,b]$, where functional $J$ is given by \eqref{funct1}.
\end{theorem}

\begin{proof}
Observe that
\begin{align*}
J(x^*+v,T^*+\triangle T)&-J(x^*,T^*)=\int_a^{T^*+\triangle T}L[x^*+v](t)\,dt-\int_a^{T^*}L[x^*](t)\,dt\\
&=\int_a^{T^*}\left[L[x^*+v](t)-L[x^*](t)\right]\,dt+\int_{T^*}^{T^*+\triangle T}L[x^*+v](t)\,dt\\
&\geq\int_a^{T^*}\left[\partial_2 L[x^*](t) \cdot v(t)+\partial_3 L[x^*](t)  \cdot \LD v(t)\right]dt+\triangle T\cdot L[x^*+v](c)=\star,
\end{align*}
for some $c$ between $T^*$ and $T^*+\triangle T$. Integrating by parts, we get
\begin{align*}
\star&=\int_a^{T^*}\left[\partial_2 L[x^*](t)+{D_{T^*-}^{\a,\p}}\left(\frac{\partial_3 L[x^*](t)}{\p'(t)}\right)\p'(t)\right]v(t)\,dt\\
& \quad+ \left[{I_{T^*-}^{1-\a,\p}}\left(\frac{\partial_3 L[x^*](t)}{\p'(t)}\right)\cdot v(t)\right]_{t=a}^{t=T^*} +\triangle T\cdot L[x^*+v](c)\\
& =\triangle T\cdot L[x^*+v](c).
\end{align*}
Since $L[x^*](T^*)=0$ and the map $t\mapsto L[x^*+v](t)$ is continuous, given $\epsilon>0$, there exists some $\delta>0$ such that
$$\|(v,\triangle T)\|<\delta \Rightarrow \left| L[x^*+v](c)\right|<\frac{\epsilon}{|\triangle T|}$$
(the case $\triangle T=0$ is obvious). Thus,
$$J(x^*+v,T^*+\triangle T)-J(x^*,T^*)\geq-\epsilon.$$
\end{proof}

We remark that we can not conclude that
$$J(x^*+v,T^*+\triangle T)-J(x^*,T^*)\geq0,$$
under these assumptions.
For example, if we consider
$$J(x,T):=\int_0^T(1-t)\,dt,$$
with the initial condition $x(0)=x_0\in\mathbb R$, it is easy to verify that $L$ is convex and that $(x^*,T^*)=(x^*,1)$, where $x^*\in C^1[a,b]$ is an arbitrary function, verifies Eqs  \eqref{ELeq} and \eqref{TransCon}. However, we have the following:
$$J(x^*+v,1+\triangle T)-J(x^*,1)=-\frac12(\triangle T)^2,$$
which is a negative number when $\triangle T\not=0$.

\subsection{Examples}

\begin{example} Consider the functional
$$J(x,T):=\int_0^T \left[\left({^CD_{0+}^{\a,\p}} x(t)-\frac{(\p(t)-\p(0))^{1-\a}}{\Gamma(2-\a)}\right)^2+t^2-1\right]\,dt$$
subject to the restriction $x(0)=0$. From Theorem \ref{teo:main}, the necessary condition that every minimizer of the functional must fulfill is the following
$${D_{T-}^{\a,\p}}\left(2\frac{{^CD_{0+}^{\a,\p}} x(t)-\frac{(\p(t)-\p(0))^{1-\a}}{\Gamma(2-\a)}}{\p'(t)}\right)\p'(t)=0,$$
for all $t\in[0,T]$, and also the two next transversality conditions
$$\left\{
\begin{array}{l}
{I_{T-}^{1-\a,\p}}\left(2\frac{{^CD_{0+}^{\a,\p}} x(t)-\frac{(\p(t)-\p(0))^{1-\a}}{\Gamma(2-\a)}}{\p'(t)}\right)=0\\
L[x](t)=0.
\end{array}
\right.$$
must be meet at  $t^=T^*$. Once
$${^CD_{0+}^{\a,\p}} \left(\p(t)-\p(0)\right)=\frac{(\p(t)-\p(0))^{1-\a}}{\Gamma(2-\a)},$$
we see that the pair
$$(x^*,T^*):=(\p(t)-\p(0),1)$$
satisfies all three conditions. Also, using the Legendre condition, a local minimizer for the functional must verify the condition
$$\partial^2_{33}L[x](t)\geq0, $$
which for our example is verified, since $\partial^2_{33}L[x](t)=2$. We remark that, for every curve $x$ and for every endpoint $T$, the inequality
$$J(x,T)\geq\int_0^T (t^2-1)\,dt\geq \int_0^1 \left[t^2-1\right]\,dt=-\frac23,$$
holds, and that
$$J(x^*,T^*)=-\frac23,$$
and so $(x^*,T^*)$ is actually the global minimizer of $J$.
\end{example}

\begin{example} For our second example, let

$$J(x,T):=\int_0^T \left[\left({^CD_{0+}^{\a,\p}} x(t)\right)^2+\left(\frac{(\p(t)-\p(0))^{1-\a}}{\Gamma(2-\a)}\right)^2+t^2-1\right]\,dt$$
subject to the restrictions $x(0)=0$ and
$$\int_0^T \left[{^CD_{0+}^{\a,\p}} x(t)\cdot\frac{(\p(t)-\p(0))^{1-\a}}{\Gamma(2-\a)}\right]\,dt=\Phi(T),$$
where the function $\Phi$ is defined as
$$\Phi(T):=\int_0^T \left[\frac{(\p(t)-\p(0))^{1-\a}}{\Gamma(2-\a)}\right]^2\,dt.$$
Considering $\lambda=-2$, the augmented function is given by
$$F(x)=\left({^CD_{0+}^{\a,\p}} x(t)-\frac{(\p(t)-\p(0))^{1-\a}}{\Gamma(2-\a)}\right)^2+t^2-1,$$
and as we have seen in the previous example, the function
$$x^*=\p(t)-\p(0)$$
satisfies the two first necessary conditions, as stated in Theorem \ref{teo:isop}. Using the second transversality condition of Theorem \ref{teo:isop}, we determine the optimal time $T^*$ by solving the equation
$${T^*}^2-1=-2.\left[\frac{(\p(T^*)-\p(0))^{1-\a}}{\Gamma(2-\a)}\right]^2.$$
\end{example}

\begin{example}
For our next example, we determine the best fractional order to minimize a given functional, using the necessary conditions given by Theorem \ref{FracOrder}.
Let
$$J(x,T,\a):=\int_a^T \Bigg[\frac{(\p(t)-\p(a))^\a}{2\Gamma(\a+2)}\cdot(\LD x(t))^2-(\p(t)-\p(a))^{\a+1}\cdot \LD x(t)+20\Gamma(\a+2)\Bigg]dt,$$
subject to the initial condition $x(a)=0$.
If we define $x^*(t):= (\p(t)-\p(a))^{\a+1}$, then $\LD x^*(t)=\Gamma(\a+2)(\p(t)-\p(a))$ and so $\partial_3 L[x^*](t)=0$, for all $t$.
Thus, the three next conditions
$$\partial_2 L[x^*](t)+{D_{T^*-}^{\a,\p}}\left(\frac{\partial_3 L[x^*](t)}{\p'(t)}\right)\p'(t)=0, \quad t\in[a,T^*],$$
$${I_{T^*-}^{1-\a,\p}}\left(\frac{\partial_3 L[x^*](t)}{\p'(t)}\right)=0, \quad t=T^*$$
and
$$\int_a^{T^*}\partial_3 L[x^*](t)\cdot \Lambda_t'(\a^*)\,dt=0,$$
are verified for all $T^*$ and $\a^*$. It remains to solve the equation $L[x^*](T^*)=0$. Since
\begin{equation}\label{auxFracOr}L[x^*](t)=0 \Leftrightarrow \p(t)=\p(a)+40^{1/(\a+2)},\end{equation}
we establish a relation between $T^*$ and $\a^*$ given by condition \eqref{auxFracOr}.
Observe that
$$J(x^*,T^*,\a)=\int_a^{T^*} \Gamma(\a+2)\left[20-\frac{(\p(t)-\p(a))^{\a+2}}{2}\right]\,dt,$$
and if we differentiate with respect to $\a$, and then  put the resulting function equal to zero, we get
\begin{equation}\label{auxFracOr1}\int_a^{T^*} \Gamma(\a+2)\left[\Psi(\a+2)\left(20-\frac{(\p(t)-\p(a))^{\a+2}}{2}\right)
-\frac{\ln(\p(t)-\p(a))}{2}\p'(t)(\p(t)-\p(a))^{\a+2}\right]\,dt=0,\end{equation}
where $\Psi$ denotes the digamma function. For example, for $\p_1(t)=t$ and $\p_2(t)=\sqrt{t+1}$, with initial point $a=0$,
we  have $T_1^*=40^{1/(\a^*+2)}$ and $T_2^*=1+40^{1/(\a^*+2)}$, respectively. The graphs of the functions
$$\a\mapsto J(x^*,T^*,\a), \quad \a\in(0,1),$$
are given in Figure \ref{FracOrder5}.
\begin{figure}[h]
  \centering
  \subfigure[Graph with respect to $\p_1$.]{\includegraphics[width=6cm]{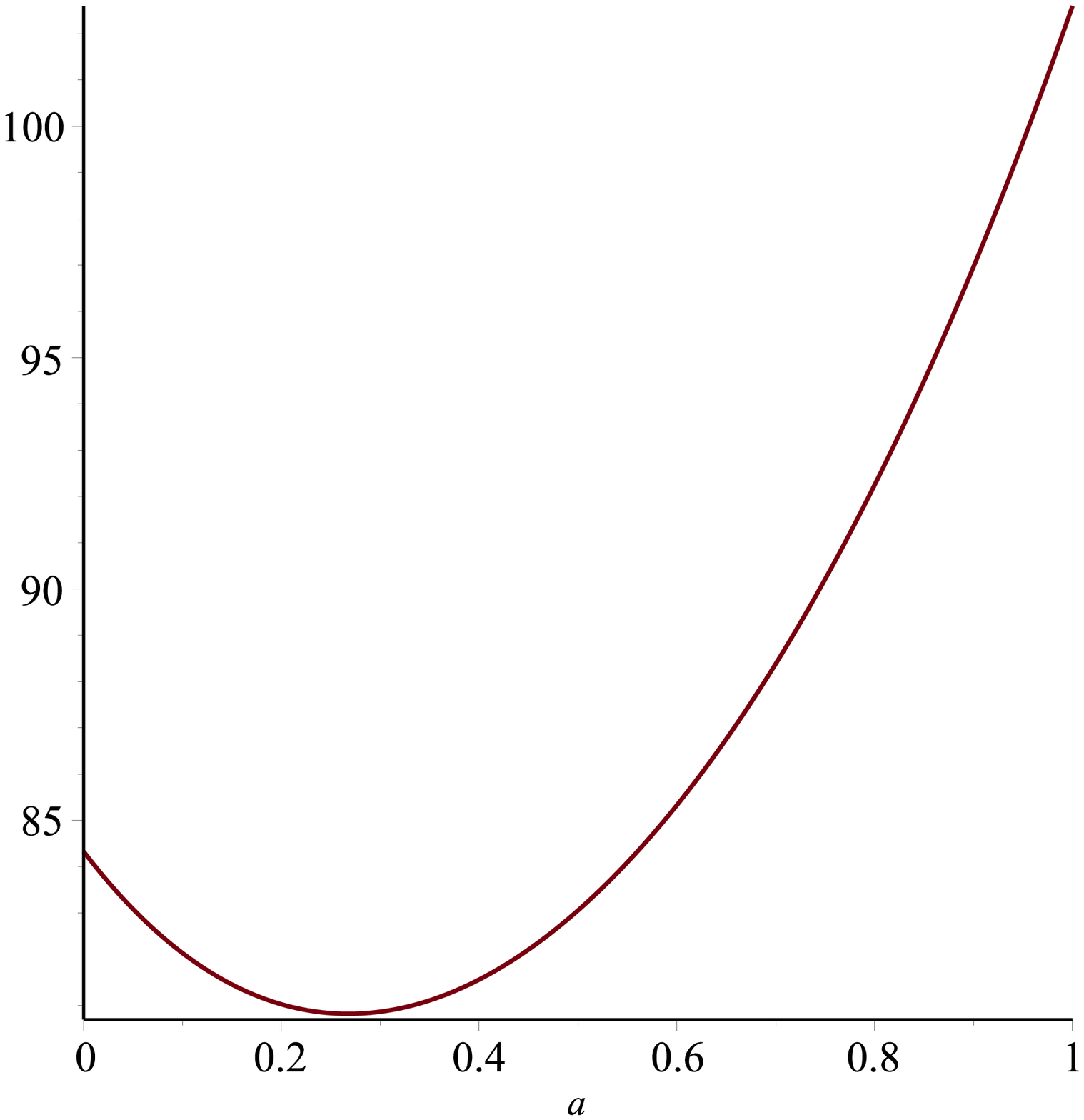}}
\subfigure[Graph with respect to $\p_2$.]{\includegraphics[width=6cm]{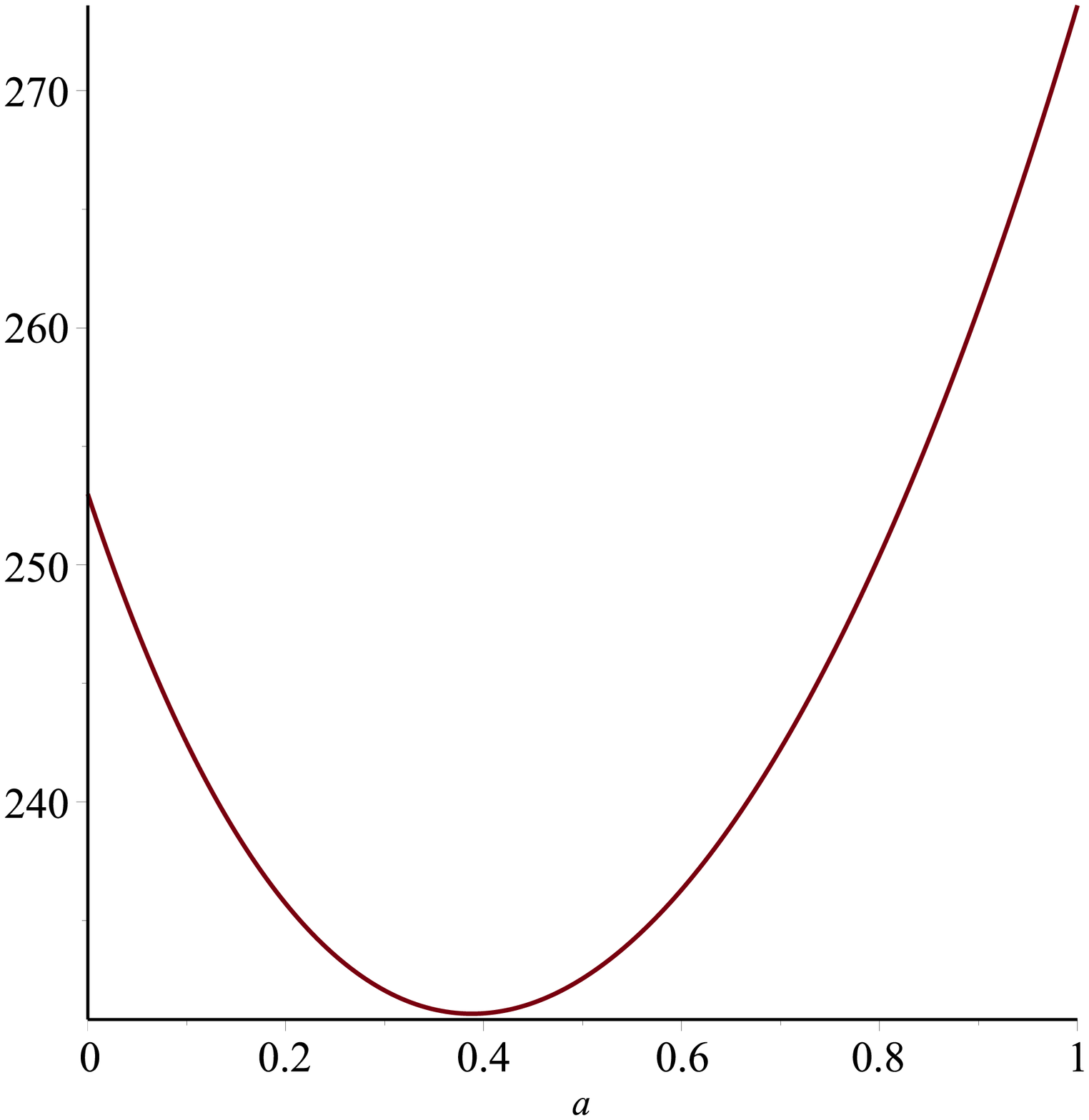}}
  \caption{Best fractional order.}\label{FracOrder5}
\end{figure}
Solving numerically Eq. \eqref{auxFracOr1}, we obtain an approximation of the fractional order that minimizes the problem:
$$\a_1\approx 0.2677 \quad \mbox{and} \quad \a_2\approx 0.2827.$$
\end{example}

\section*{Acknowledgments}

Work supported by Portuguese funds through the CIDMA - Center for Research and Development in Mathematics and Applications, and the Portuguese
Foundation for Science and Technology (FCT-Funda\c{c}\~ao para a Ci\^encia e a Tecnologia), within project UID/MAT/04106/2013.



\begin{thebibliography}{99}


\bibitem{AGRA1}
O. P. Agrawal,
Formulation of Euler-Lagrange equations for fractional variational problems,
J. Math. Anal. Appl. 272  (2002), no. 1, 368--379.

\bibitem{AGRA2}
O. P. Agrawal,
Fractional variational calculus and the transversality conditions,
J. Phys. A  39  (2006), no. 33, 10375--10384.

\bibitem{AGRA3}
O. P. Agrawal,
Fractional variational calculus in terms of Riesz fractional derivatives,
J. Phys. A  40   (2007), no. 24, 6287--6303.

\bibitem{Almeida1}
R. Almeida, Fractional variational problems with the Riesz--Caputo derivative,
Appl. Math. Lett.  25 (2012), 142--148.

\bibitem{AlmeidaD}
R. Almeida,
Fractional variational problems depending on indefinite integrals and with delay,
Bull. Malays. Math. Sci. Soc. (in press).

\bibitem{AlmeidaN}
R. Almeida,
A Caputo fractional derivative of a function with respect to another function (submitted)

\bibitem{Almeida2}
R. Almeida, R. A. C. Ferreira and D. F. M. Torres,
Isoperimetric problems of the calculus of variations with fractional derivatives,
Acta Math. Sci. Ser. B Engl. Ed. 32 (2012), no. 2, 619--630.

\bibitem{AlmeidaIn}
R. Almeida and A.B. Malinowska,
Generalized transversality conditions in fractional calculus of variations,
Commun. Nonlinear Sci. Numer. Simul. 18 (2013) 443--452.


\bibitem{book2}
R. Almeida, S. Pooseh and D.F.M. Torres,
Computational Methods in the Fractional Calculus of Variations,
Imp. Coll. Press, London, 2015.

\bibitem{Atanackovic}
T. M. Atanackovi\'c, S. Konjik\ and\ S. Pilipovi\'c,
Variational problems with fractional derivatives: Euler-Lagrange equations,
J. Phys. A 41  (2008), no. 9, 095201, 12 pp.

\bibitem{Baleanu2}
D. Baleanu and S. I. Muslih, Lagrangian formulation of classical fields within Riemann--Liouville fractional derivatives,
Phys. Scripta 72 (2005), no. 2--3, 119--121.

\bibitem{BaleanuD1}
D. Baleanu, T. Maaraba and F. Jarad,
Fractional variational principles with delay,
J. Phys. A  41 (2008), no. 31, 315403, 8 pp.

\bibitem{Baleanu1}
D. Baleanu,
New applications of fractional variational principles,
Rep. Math. Phys. 61 (2008), no. 2, 199--206.

\bibitem{baleanu2}
M. A. E. Herzallah\ and\ D. Baleanu,
Fractional-order Euler-Lagrange equations and formulation of Hamiltonian equations,
Nonlinear Dynam. 58 (2009), no. 1--2, 385--391.

\bibitem{brock}
W. A. Brock,
On existence of weakly maximal programmes in a multi-sector economy.
Rev. Econom. Stud. 37 (1970) 275--280.

\bibitem{Brunt}
B. van Brunt, The calculus of variations, Universitext, Springer, New York, 2004.

\bibitem{BaleanuD2}
F. Jarad, T. Abdeljawad and D. Baleanu,
Fractional variational principles with delay within Caputo derivatives.
Rep. Math. Phys. 65 (2010), no. 1, 17--28.

\bibitem{Lazo}
M. J. Lazo and D. F. M. Torres,
The Legendre condition of the fractional calculus of variations.
Optimization 63 (2014) 1157--1165.

\bibitem{Agnieszka1}
A. B. Malinowska\ and\ D. F. M. Torres,
Generalized natural boundary conditions for fractional variational problems in terms of the Caputo derivative,
Comput. Math. Appl. 59 (2010),  no. 9, 3110--3116.

\bibitem{book1}
A. B. Malinowska and D. F. M. Torres,
Introduction to the fractional calculus of variations,
Imp. Coll. Press, London, 2012.

\bibitem{Odziehjewicz}
T. Odziehjewicz,
Generalized fractional isoperimetric problem of several variables,
Discrete Contin. Dyn. Syst. Ser. B 19 (2014), no. 8, 2617--2629.

\bibitem{withTatiana:Basia}
T. Odzijewicz, A. B. Malinowska and D. F. M. Torres,
Fractional variational calculus with classical and combined Caputo derivatives,
Nonlinear Anal. 75 (2012), no. 3, 1507--1515.

\bibitem{Richard}
J.-P. Richard,
 Time-delay systems: an overview of some recent advances and open problems,
Automatica J. IFAC 39 (2003), no. 10, 1667--1694.

\bibitem{Salamon}
 D. Salamon,
On controllability and observability of time delay systems,
IEEE Trans. Automat. Control 29 (1984), no. 5, 432--439.

\bibitem{Samko}
S. G. Samko, A. A. Kilbas and O. I. Marichev,
Fractional integrals and derivatives, translated from the 1987 Russian original,
Gordon and Breach, Yverdon, 1993.


\end{thebibliography}
\end{document}